\documentclass[12pt]{amsart}

\usepackage[margin=1.15in]{geometry}
\usepackage{amscd,amssymb, amsmath, wasysym, mathrsfs, mathtools, color}
\usepackage{graphicx}
\usepackage[all, cmtip]{xy}
\usepackage{url}
\usepackage[pagebackref=true]{hyperref}

\theoremstyle{plain}
\newtheorem{theorem}{Theorem}[section]
\newtheorem{prop}[theorem]{Proposition}

\newtheorem{cor}[theorem]{Corollary}

\theoremstyle{definition}

\newtheorem{defn}[theorem]{Definition}
\newtheorem{rmk}[theorem]{Remark}

\newtheorem{ex}[theorem]{Example}
\newtheorem*{ex*}{Example}

\theoremstyle{Question}

\newcommand\sA{{\mathcal A}}
\newcommand\sB{{\mathcal B}}

\newcommand\sV{{\mathcal V}}
\newcommand\sL{\mathcal{L}}

\newcommand{\sS}{\mathcal{S}}

\newcommand\qq{{\mathbb{Q}}}
\newcommand\zz{{\mathbb{Z}}}
\newcommand\rr{{\mathbb{R}}}
\newcommand\cc{{\mathbb{C}}}

\newcommand\CN{{\mathbb{C}^{n}}}

\newcommand{\rank}{{\rm rank}}

\newcommand{\ubul}{{\,\begin{picture}(-1,1)(-1,-3)\circle*{2}\end{picture}\ }}

\DeclareMathOperator{\Supp}{Supp}                
\DeclareMathOperator{\id}{id}                    

\DeclareMathOperator{\homo}{Hom}

\DeclareMathOperator{\spec}{Spec}

\DeclareMathOperator{\im}{Im}

\title{The monodromy theorem for compact K\"ahler manifolds and smooth quasi-projective varieties}

\author{Nero Budur}
\address{N. Budur: KU Leuven, Department of Mathematics,
Celestijnenlaan 200B, B-3001 Leuven, Belgium} 
\email {nero.budur@kuleuven.be}

\author{Yongqiang Liu}
\address{Y. Liu: KU Leuven, Department of Mathematics,
Celestijnenlaan 200B, B-3001 Leuven, Belgium}
\email {yongqiang.liu@kuleuven.be}

\author{Botong Wang}
\address{B. Wang: Department of Mathematics, University of Wisconsin, 480 Lincoln Drive, Madison, WI 53706, USA} 
\email {bwang274@wisc.edu}

\begin{document}

\date{}


\begin{abstract} 
Given any connected topological space $X$, assume that there exists an epimorphism $\phi: \pi_1(X) \to \zz$. The deck transformation group $\zz$ acts on the associated infinite cyclic cover $X^\phi$ of $X$, hence on the homology group $H_i(X^\phi, \cc)$. This action induces a linear automorphism on the torsion part of the homology group  as a module over the Laurent ring $\cc[t,t^{-1}]$, which is a finite dimensional $\cc$-vector space. We study the sizes of the Jordan blocks of this linear automorphism. When $X$ is a compact K\"ahler manifold, we show that all the Jordan blocks are of size one. When $X$ is a smooth complex quasi-projective variety, we give an upper bound on the sizes of the Jordan blocks, which is an analogue of the Monodromy Theorem for the local Milnor fibration.  
\end{abstract}
\maketitle
\section{Introduction}
\subsection{Motivation}
Let us first recall  the well-known Monodromy Theorem in singularity theory. 
Let $f: (\CN,0) \to (\cc, 0)$ be a germ of analytic function. Let $B_\epsilon$ be a small open ball at the origin in $\CN$. We call $X= B_\epsilon \setminus f^{-1}(0)$ the local complement of the germ $\{f=0\}$. Let  $D_{\delta} \subset \cc$ be a disc around the origin with $0< \delta\ll \epsilon$. Set $D_\delta^{*}= D_\delta \setminus \{0\}$. Then there exists a Milnor fibration  $f: X \cap f^{-1}(D^*_\delta)\to D_\delta^{*} $  with Milnor fibre $F_{f}$.  Let $h: F_{f} \to F_{f}$ be the monodromy homemorphism associated with going once anti-clockwise along the boundary $\partial D_{\delta}$. It induces the corresponding automorphisms: $$h_i: H_i (F_f,\cc) \to H_i (F_f,\cc).$$
Note that $F_{f}$ has the homotopy type of a finite $(n-1)$-dimensional CW complex. So the only interesting homology group $H_i (F_f,\cc)$ appears in  the range $0\leq i \leq n-1$.

\begin{theorem}[Monodromy Theorem] \label{monodromy} \cite[Theorem 3.1.20]{D1} With the above assumptions and notations, we have that 
\begin{enumerate}
\item[(1)] the eigenvalues of $h_{i}$ are all roots of unity for all $i$.
\item[(2)] the sizes of the blocks in the Jordan normal form of $h_i$ are at most $i+1$. 
\end{enumerate}
\end{theorem}
Examples of Malgrange \cite{Mal} show that the bounds on the sizes of the Jordan blocks are sharp.

The monodromy contains a lot of information about the topology of the singularity. The Monodromy Theorem shows that  the local complement of the germ has very  special topology. The aim of this paper is try to prove an analogue of the Monodromy Theorem with $X$ either a compact K\"ahler manifold or a smooth quasi-projective variety.

\subsection{Infinite cyclic cover}
Since in general there does not exist a circle fibration for $X$ either a compact K\"ahler manifold or a smooth quasi-projective variety, one can not define the monodromy action on the fibres any more. The alternative is the $\zz$-infinite cyclic cover  of $X$ associated to an epimorphism $\pi_{1}(X) \to \zz$ with the corresponding deck transformation action.

Let $X$ be a connected topological space of finite homotopy type. Let $\phi: \pi_1(X)\to \zz$ be an epimorphism. Consider the infinite cyclic cover $X^{\phi}$ of $X$ defined by $ker(\phi)$. The deck transformation action on $X^\phi$ induces $\zz$ actions on $H_i(X^\phi, \cc)$. Let $R:=\cc[t, t^{-1}]$ be the one variable Laurent polynomial ring. Then the $\zz$ action on each of the above homology group is equivalent to an $R$-module structure. Since $X$ is a homotopy equivalent to a finite CW-complex, $H_i(X^\phi, \cc)$ is a finitely generated $R$-module for any $i$.

Note that $R$ is a principal ideal domain.  Denote the torsion part of $H_i(X^\phi, \cc)$ as an $R$-module by $T_i(X, \phi)$. Then $T_i(X, \phi)$ is a finite dimensional $\cc$-vector space. The deck transformation, or equivalently multiplication by $t$ as an $R$-module, acts on $T_i(X, \phi)$ as a linear automorphism. Let $$T_i(X, \phi)=\bigoplus_{\lambda\in \cc^*} T_i^\lambda(X, \phi)$$ be the decomposition according to the support, or equivalently as generalized eigenspaces. Here the support is defined as
$$ \Supp (T_i(X, \phi)):= \{ \lambda \in \cc^* \mid \exists  a \in T_i(X, \phi) \text{ and } a\neq 0 \text{ such that } (t-\lambda) \cdot a=0 \} ,$$
and it is a finite set.  

Denote by $\sS_{\lambda}(T_i(X, \phi))$  the maximal size of Jordan blocks associated to the eigenvalue $\lambda$. Equivalently, in the language of $R$-modules,  $\sS_{\lambda}(T_i(X, \phi))$ is the minimal non-negative integer $m$ such that $(t-\lambda)^m \cdot T_i^\lambda(X, \phi)=0$.
Then $\sS_{\lambda}(T_i(X, \phi))=0$ if and only if $\lambda \notin \Supp (T_i(X, \phi))$ by convention.
On the other hand, $\sS_{\lambda}(T_i(X, \phi)) = 1$ if and only if $T_i^\lambda(M, \phi)$ is a nontrivial semi-simple $R$-module.

\begin{ex}\label{example1}
Let $f: X\to S^1$ be a fiber bundle with connected fiber $F$, which is a finite CW-complex. Let $\phi=f_*: \pi_1(X)\to \pi_1(S^1)=\zz$, where the connected assumption about the fiber $F$ implies that this map is surjective. Then the infinite cyclic cover $X^\phi$ is homeomorphic to $F\times \rr$, and hence homotopy equivalent to $F$. The deck transformation action on $H_i(X^\phi, \cc)$ is isomorphic to the monodromy action on $H_i(F, \cc)$. In particular, $H_{i}(X^\phi,\cc)\cong H_{i}(F, \cc)$ is a torsion $R$-module for all $i$. It is clear that $\Supp (T_i(X, \phi))$ and $\sS_\lambda(T_i(X,\phi))$ coincide with the corresponding notions associated to the monodromy action on  $H_i(F,\cc)$. 
\end{ex}
\begin{rmk}
Notice that both the Laurent polynomial ring $R$ and the Alexander module $H_i(X^\phi, \cc)$ can be defined over $\zz$. Therefore, for any $\lambda, \lambda'\in \cc^*$ in the same $\mathrm{Gal}(\cc/\qq)$ orbit, $T_i^\lambda(X, \phi)$ is non-canonically isomorphic to $T_i^{\lambda'}(X, \phi)$ as $\mathbb{C}$-vector spaces, and their Jordan blocks have the same sizes. Moreover, if $\lambda$ is a transcendental number, then $T_i^\lambda(X, \phi)=0$.
\end{rmk}

A structure theorem has been proved for the jumping loci of $X$, when $X$ is either the complement in a small ball of  a complex analytic function germ (by Budur-Wang \cite{bw3}), called the local complement as above, a smooth complex quasi-projective variety (by Budur-Wang \cite{bw1}) or a compact K\"ahler manifold (by Wang \cite{w}). This structure theorem implies the following claim directly, which can be viewed as a generalization of the first part of the Monodromy Theorem \ref{monodromy} for all these three cases. 

\begin{prop} \label{torsion} Let $X$ be either a local complement, a smooth complex quasi-projective variety or a compact K\"ahler manifold. Then for any epimorphism $\phi: \pi_1(X)\to \zz$, the eigenvalues associated to the $t$-action on $ T_i(X, \phi)$ are roots of unity for any $i$.
\end{prop}

 \begin{rmk} If $X$ is in one of the three cases as in Proposition \ref{torsion} with complex dimension $n$, then $X$ has the homotopy type of a finite $2n$-dimensional CW complex. It implies that $H_{i}(X^{\phi},\cc)=0$ for $i>2n$, and $H_{2n}(X^{\phi},\cc)$ is a free $R$-module. Hence the only interesting  $R$-modules $T_{i}(X,\phi)$ appear in the range $0\leq i\leq 2n-1$. 
  \end{rmk}

\subsection{Compact K\"ahler manifolds}
Our first main result is that, for compact K\"ahler manifolds, all the Jordan blocks in $T_i(X, \phi)$ are of size one, i.e., $T_i(X, \phi)$ is a semi-simple $R$-module. 
\begin{theorem}\label{main1}
Let $X$ be a connected compact K\"ahler manifold. Then for any epimorphism $\phi: \pi_1(X)\to \zz$,   $T_i(X, \phi)$ is a semi-simple $R$-module for any $i$.
\end{theorem}
Proposition \ref{torsion} allows us to reduce the proof to the case when $\lambda=1$. When $\lambda=1$, using the following proposition, we reduce to the fact that the real homotopy type of a compact K\"ahler manifold is formal \cite{dgms}. 
\begin{prop}\label{formal}
Let $X$ be a smooth real manifold, which is of finite homotopy type. Suppose that $X$ is formal. Then for any epimorphism $\phi: \pi_1(X)\to \zz$, the eigenvalue $1$ part of $T_i(X, \phi)$ is a semi-simple $R$-module, i.e., $(t-1)\cdot T_i^1(X, \phi)=0$. 
\end{prop}
This proposition generalizes results of Papadima-Suciu \cite{ps} and Fern\'andez-Gray-Morgan \cite{fgm}.

\subsection{Smooth quasi-projective varieties}

Our second main result is a global analogue of the second part of the Monodromy Theorem \ref{monodromy}. 
\begin{theorem}\label{main2}
Let $X$ be a smooth complex quasi-projective variety with complex dimension $n$. Fix an epimorphism $\phi: \pi_1(X)\to \zz$.  Then for any $i$ and any $\lambda \in \cc^*$,
$$\sS_\lambda(T_i(X, \phi))\leq \min \{i+1, 2n-i\}.$$
\end{theorem}

Similar to the proof of Theorem \ref{main1}, by Proposition \ref{torsion}, we first reduce to the case $\lambda=1$. When $\lambda=1$, we prove the theorem using Morgan's Gysin model \cite{m}. 

\begin{rmk}  Assume that $f: \CN \to \cc$ is a reduced polynomial map.  Set $X= \CN \setminus f^{-1}(0)$. Then $f$ induces a surjective map $$\phi = f_* : \pi_1(X) \to \pi_1(\cc^*)\cong \zz.$$
 Note that in this case $X$ is an affine variety, hence it has the homotopy type of $n$-dimensional CW-complex. So the only interesting $T_i(X,\phi)$ appears in the range $0\leq i\leq n-1$. Proposition \ref{torsion} and Theorem \ref{main2} give the same results for the $t$-action on $T_i(X,\phi)$ as the one in the Monodromy Theorem \ref{monodromy}.
\end{rmk}

There exists (see \cite[page 8]{A}) mixed Hodge structure  on  $H^1(X,\cc)$ with an increasing weight filtration : $$0 =W_0 (H^1(X,\qq))  \subset W_1 (H^1(X,\qq)) \subset W_2 (H^1(X,\qq))=H^1(X,\qq), $$
and a decreasing Hodge filtration 
$$H^1(X,\cc)=F^0(H^1(X,\cc))\supset F^1(H^1(X,\cc))\supset F^2(H^1(X,\cc))=0,$$
such that $H^1(X,\cc)$ is the direct sum of $W_1(H^1(X,\cc))$ and $F^1(H^1(X,\cc))\cap \overline{F^1(H^1(X,\cc))}$. The epimorphism $\phi: \pi_1(X)\to \zz$ induces a monomorphism in cohomology $$\phi^{\ast}: \cc \hookrightarrow H^{1}(X,\cc),$$ which gives a $1$-dimensional $\cc$-vector subspace $\im(\phi^\ast)$ in $H^{1}(X,\cc)$.    We say that $\phi$ is pure with weight one , if $\im(\phi^\ast) \subset W_1(H^1(X,\cc))$. We say that $\phi$ is pure with weight two, or equivalently of type (1,1), if $\im(\phi^*)\subset F^1(H^1(X,\cc))\cap \overline{F^1(H^1(X,\cc))}$.  If neither of the above is true, we say that $\phi$ has a mixed type.

\begin{prop}\label{pure}  Let $X$ be a smooth complex quasi-projective variety with  complex dimension $n$. Fix an epimorphism $\phi: \pi_1(X)\to \zz$.    If $\phi$ is pure with weight two, then for any $\lambda \in \cc^*$,
$$\sS_\lambda(T_i(X, \phi))\leq \min\{i+1, 2n-i-1\}.$$
In particular, $H_{2n-1}(X^\phi,\cc)$ is a free $R$-module.
\end{prop}
Next, we recall a slightly different monodromy theorem and compare it with Proposition \ref{pure}.
 Let  $f: Y \to S$ be an algebraic map of algebraic varieties such that $\dim_\cc Y=n$ and $\dim_\cc S = 1$.  Assume that $S$ is a smooth curve. There exists a finite set $B\subset S$ such that $f$ is a locally trivial fibration over $S\setminus B$ with generic fibre $F$. We assume that the fiber $F$ is connected.  For any $b\in B$, let $D_b$ be a small enough disc in $S$ centred at $b$.  Let $h_b: F \to F$ be the monodromy homemorphism associated with going once anti-clockwise along the boundary $\partial D_b$. Note that $H^1(\partial D_b, \cc)\cong \cc$ can be viewed as a vector space with pure Hodge structure type $(1,1)$, hence the induced epimorphism $f_* : \pi
 _1(f^{-1}(\partial D_b)) \to \pi_1(\partial D_b) \cong \zz$  is pure with weight 2.  Using the limit mixed Hodge structures on $H^i(F,\cc)$ (hence also for $H_i(F,\cc)$), one can prove that the Jordan blocks associated to the monodromy action on $H_i(F, \cc)$ are of size at most $\min\{i+1, 2n-i-1\}$, see \cite[Corollary 2.1]{D2}. This upper bound coincides with the one in Proposition \ref{pure}. 

\subsection{Application}

When the epimorphism $\phi: \pi_1(X)\to \zz$ can be realized as a bundle map as in Example \ref{example1}, the results of Theorem \ref{main1} and Theorem \ref{main2} provide restrictions for the monodromy action on the homology of the fibers. In this case, we get a refined version of \cite[Theorem 1.3]{qw}.
\begin{cor} \label{bundle}
Let $f: X\to S^1$ be a fiber bundle with connected fiber $F$, where the homology groups $H_i(F,\cc)$ are finite dimensional vector spaces for all $i$. 
\item[(a)] If $X$ is homotopy equivalent to a compact K\"ahler manifold, then 
\begin{enumerate}
\item[(i)] the eigenvalues of the monodromy action on $H_i(F, \cc)$ are roots of unity; 
\item[(ii)] the monodromy action is semi-simple;
\end{enumerate}
\item[(b)] If $X$ is homotopy equivalent to a smooth quasi-projective variety of complex dimension $n$, then 
\begin{enumerate}
\item[(i)] the eigenvalues of the monodromy action on $H_i(F, \cc)$ are roots of unity; 
\item[(ii)] the Jordan blocks of the monodromy action on $H_i(F, \cc)$ are of size at most $\min\{i+1, 2n-i\}$.
\item[(iii)] this upper bound associated to $H_i(F, \cc)$ can be improved to $\min\{i+1, 2n-i-1\}$, when the induced epimorphism $f_*: \pi_1(X) \to \pi_1(S^1)$ is pure with weight two.
\end{enumerate} 
\end{cor}

\begin{rmk}  Let $f:X \to \cc^*$ be a proper smooth map, where $X$ is a smooth quasi projective complex variety.  Deligne's decomposition theorem (\cite[Theorem 1.2.1]{dem}) shows that the induced monodromy action on the homology groups of the fiber $H_*(F,\cc)$ is always semi-simple. It is interesting to compare this claim with Corollary \ref{bundle}.  
\end{rmk}


\medskip

\noindent{\it Convention:} 
Unless otherwise specified, all homology and cohomology groups will be assumed to have $\cc$-coefficients.

\medskip

\textbf{Acknowledgments.} 
We thank Lizhen Qin for many helpful discussions. N. Budur and Y. Liu were partially supported by  a FWO grant, a KU Leuven OT grant, and a Flemish Methusalem grant.

\section{Preliminaries}

\subsection{Alexander modules}
Let $X$ be a connected topological space of finite homotopy type. Let $\phi: \pi_1(X)\to \zz$ be an epimorphism. Consider the infinite cyclic cover $X^{\phi}$ of $X$ defined by $ker(\phi)$. The deck transformation action on $X^\phi$ induces $\zz$ actions on $H_i(X^\phi, \cc)$.  Then, under the deck group action, every homology group $H_{i}(X^{\phi},\cc)$ becomes a $R$-module. 

\begin{defn}  The $R$-module $ H_{i}(X^{\phi},\cc)$  is called the \textit{$i$-th Alexander module} of $X$ associated to the epimorphism $\phi: \pi_1(X)\to \zz$. 
\end{defn}

Consider the local system $\sL^\phi $ on $X$ with stalk $R$, and representation of the fundamental group defined by the composition: 
$$ \pi_{1}(X) \overset{\phi}{\rightarrow} \zz   \rightarrow   Aut(R),$$
with the second map being given by  $1_{\zz}\mapsto t$. 
Here $t$ is the automorphism of $R$ given by multiplication by $t$.

By choosing fixed lifts of the cells of $X$ to $X^{\phi}$, we obtain a free basis for the chain complex  $C_{\ubul}(X^{\phi},\cc)$ as $R$-modules. The homology of $\sL^\phi$ can be computed by the singular chain complex  $C_\ubul(X^{\phi},\cc)$ as  $R$-modules.
 In fact, we have the following $R$-module isomorphisms:
\begin{equation}\label{max1} 
 H_{i}(X,\sL^\phi)\cong H_{i}(X^{\phi})\cong H_i\left(C_\ubul(X^{\phi},\cc) \right) \text{   for all   } i.
\end{equation}

Let $\overline{\sL^\phi}= \mathit{Hom}(\sL, R_X)$ be the dual complex of $\sL$, which is the local system obtained from $\sL^\phi$ by composing all $R$-module structures with the involution $t \mapsto t^{-1}$. Here $R_X$ is the constant sheaf on $X$ with stalk $R$. Consider the dual complex $\homo _R(C_\ubul(X^\phi, \cc), R)$. Then it follows that 
$$ H^i(X, \overline{\sL^\phi}) \cong H^i \left(\homo _R(C_\ubul(X^\phi, \cc), R)\right) $$
 Using the Universal Coefficient Theorem (e.g., see \cite[Theorem 1.4.5]{D1}), we obtain:  
 \begin{equation}\label{uct}
 H_{i}(X^\phi) \cong H_i(X,\sL^\phi)   \cong {\rm Free} \left( H^{i}(X, \overline{ \sL^\phi}) \right)\oplus {\rm Tor} \left( H^{i+1}(X,\overline{\sL^\phi})\right),
 \end{equation}
 where $\mathrm{Free}$ (resp. $\mathrm{Tor}$) means taking the free (resp. torsion) part as $R$-modules.

Note that \begin{equation} \label{max2} 
H^{i}(X,\overline{\sL^\phi})\cong \overline{H^i(X,\sL^\phi)},
\end{equation}
where both overlines denotes the the new $R$-module structure by composing with the involution $t \rightarrow t^{-1}$.

\subsection{The Structure Theorem}

First, we recall the definition of the jumping loci. Let $X$ be a connected finite CW-complex with $\pi_{1}(X)=G$. Then the group of $\cc$-valued  characters, $ \homo(G,\cc^{\ast})$, is a commutative affine algebraic group. Each character $\rho \in \homo(G,\cc^{\ast})$ defines a rank one local system on $X$, denoted by $L_{\rho}$.
The homology jumping loci of $X$ are defined as follows:
$$\sV_i^k(X)=\lbrace \rho\in \homo(G,\cc^{\ast}) \mid \dim_{\cc} H_{i}(X, L_{\rho})\geq k \rbrace.$$ 

\begin{theorem}[Structure Theorem, \cite{bw1,bw3,w}]  Let $X$ be either a local complex complement, a smooth complex quasi-projective variety or a compact K\"ahler manifold. Then each $\sV^{i}_{k}(X)$ is a finite union of torsion translated sbutori of $\homo(G,\cc^*)$.
\end{theorem}
As a direct application of the Structure Theorem, we prove Proposition \ref{torsion}.
\begin{proof}[Proof of Proposition \ref{torsion}]
Note that  $\zz$ and $\cc^*$ (as multiplicative group) are both abelian groups. So the map $\phi$ and any map in $\homo(G, \cc^*)$ factor through $H_1(X,\zz)$. Consider the following commutative diagram:
$$
\xymatrix{
G=\pi_1(X) \ar[r]^{ab} \ar[rd]^{\phi}   & H_1(X,\zz)  \ar[d]^{\phi_{ab}} \ar[r]  & \cc^* \\
             &     \zz    \ar[r]    & \cc^*  
}$$
where the map $ab$ denotes the abelization map from the fundamental group to its homology group, and $\phi_{ab}$ is defined such that $\phi =\phi_{ab} \cdot ab$. Note that $\phi_{ab}$ is surjective, hence it induces a monomorphism $$\phi^\#: \homo(\zz, \cc^\ast) \to \homo(G,\cc^*).$$ Denote $\mathbb{T}^\phi := \im(\phi^\#) $, which is isomorphic to $\cc^*$.

 The eigenvalues for $T_{i}(X,\phi)$ can be observed from $\sV_{i}^{k}(X)\cap \mathbb{T}^\phi$, if we choose the right number $k$. According to \cite[Theorem 4.2]{DN}, $k$ can be chosen as $$k:= \rank   H_{i}(X^{\phi})  +1,$$
 where the rank is for the free part of  $H_{i}(X^{\phi})$ as $R$-modules.  For such $k$,   the intersection $\sV_{i}^{k}(X)\cap \mathbb{T}^\phi$ consists of at most finitely many points.  The structure theorem implies that these finite points are all torsion points. Moreover, \cite[Theorem 4.2]{DN} shows that  $\Supp(T_{i}(X,\phi))$ is included in $(\phi^\#)^{-1}\left( \sV_{i}^{k}(X)\cap \mathbb{T}^\phi\right)$, hence the eigenvalues are roots of unity.  
\end{proof}


\begin{rmk}
Note that over a PID, a bounded complex of finitely generated modules is determined by its homology modules, up to quasi-isomorphism. Moreover, over a PID, a finitely generated module is determined by its fitting ideals. Therefore, over the one variable Laurent polynomial ring $R$, the homology jump ideals determine all Alexander modules, and vice versa. Next, we want to make this correspondence more explicit. 

By choosing fixed lifts of the cells of $X$ to $X^{\phi}$, we obtain a free basis for the chain complex  $C_{\ubul}(X^{\phi},\cc)$ as $R$-modules.
  We define the notion of homology jump ideal of $X$ associated to the epimorphism $\phi:\pi_1(X) \to \zz$ as follows: 
$$  J_i^k (X,\phi) := I_{rank (C_i(X^\phi))-k+1}(\partial_{i+1} \oplus \partial_i)  ， 
$$
where $I$ denotes the determinantal ideal,  $\partial_{i+1}: C_{i+1}(X^\phi) \to C_i(X^\phi) $ and $ \partial_i:  C_{i}(X^\phi) \to C_{i-1}(X^\phi) $ are the boundary maps of the chain complex.  The ideals $ J_i^k (X,\phi)$ do not depend on the choice of $C_{\ubul}(X^\phi)$. The homology jump loci ideals $ J_i^k (X,\phi)$ define Zariski-closed sub-scheme of $\spec(R)$.

In general,  let $E_\ubul$ be a bounded chain complex of $R$-modules whose homology modules are finitely generated. By a result of Mumford (see \cite[III.12.3]{Ha}), there exists a bounded  complex $F_\ubul$ of finitely generated free $R$-modules, which is a quasi-isomorphic to $E_\ubul$.    Then $J_i^k (E_\ubul)$ is defined to be $J_i^k (F_\ubul)$, and it does not depend on the choice of $F_\ubul$.

For every finitely generated $R$-module $A$, $A$ has a  decomposition, $$A \cong R^{r} \oplus (\oplus_{\lambda }  A_{\lambda}),$$ where $r$ is the rank of $A$, the second sum is over the finite set $\Supp (T(A)) \subset\cc^*$, and $A_{\lambda}$ denotes the $(t-\lambda)$ torsion part of $A$. Here $T(A)$ denotes the torsion part of $A$.

For each $A_{\lambda}\neq 0$, there exists a unique decomposition $$A_{\lambda} \cong \oplus_{j\geq 1} R/((t-\lambda)^{h_{j}(\lambda)} ),$$
where $h_1(\lambda)\geq 1,$ and $h_{1}(\lambda) \geq h_2(\lambda) \geq \cdots \geq  h_j(\lambda) \geq  \cdots  $. In fact, $h_1(\lambda)= \sS_\lambda(T(A))$.  Note that  $h_j(\lambda)=0$ for $j$ large enough.   If $A_\lambda=0$, we set $h_j(\lambda)=0$ for any $j\geq 1$.

The module $A$ can be viewed as a chain complex with only one non-zero term $A$ at degree 0. Then $ J_0^k (A)$ is $0$ or $ \prod_{\lambda \in \cc^*} \prod_{j\geq k-r} (t-\lambda)^{h_j(\lambda)}, $ depending on whether $1\leq k \leq r$ or $ r+1\leq k $.


Now we can read $J_i^k(X,\phi)$ from $H_i(X^\phi)$ and $H_{i-1}(X^\phi)$. Let $r_i$ denote the rank of $H_i(X^\phi)$, and $h_j(i,\lambda)$ denote the corresponding number for $T_i^\lambda(X,\phi)$. Here the non-zero part of $\{ h_j(\lambda)\}$ is exactly the list of all the sizes of Jordan blocks for $T_i^\lambda(X,\phi)$. Put the two sequences $\{h_j(i,\lambda)\}$ and $\{ h_j(i-1,\lambda)\}$ together and reorder them. We get a new sequence $\{ h_j^\prime(i,i-1,\lambda)\}$ in the descending order.  Then \begin{center}
 $ J_i^k (X,\phi)=\left\{ \begin{array}{ll}
0, & 1\leq k \leq r_i, \\
 \prod_{\lambda \in \cc^*} \prod_{j\geq k-r_i} (t-\lambda)^{h_j^\prime(i,i-1,\lambda)}, &  r_i+1\leq k.\\
\end{array}\right.  $
\end{center}  
More precisely, knowing $J_i^k(X,\phi)$ for all $k$ is equivalent to knowing the ranks $r_i$ and $T_i(X,\phi) \oplus T_{i-1}(X,\phi)$.

On the other hand, it is clear that if one knows $J_i^k(X,\phi)$  for all $i\leq m$, then one can recover $H_\ubul(X,\phi)$ from these data up to degree $m$.
\end{rmk}
\begin{rmk} One can consider the universal abelian cover $X^{ab}$ of $X$ associated to the Hurewicz map $ab: \pi_1(X)\to H_1(X,\zz)$. Then $C_\ubul(X^{ab},\cc)$ becomes a complex of free $\cc[H_1(X,\zz)]$-modules.  The homology jump ideal $J_i^k(X)$ can be defined similarly. 
Then $\sV_i^k(X)$ is the closed subscheme defined by the ideal $J_i^k(X)$. 
 
On the other hand, $J_i^k(X,\phi)$ is same as the image of  $ J_i^k(X) $ under the surjective ring homomorphism  $\cc[H_1(X,\zz)] \to \cc[\zz]$, which is induced by the epimorphism $\phi_{ab}: H_1(X,\zz) \to \zz$ defined in the proof of Proposition \ref{torsion}.
\end{rmk}

\subsection{Finite cover}
 Assume that $X$ is either a compact K\"ahler manifold or a smooth complex quasi-projective variety both with complex dimension $n$. Proposition \ref{torsion} shows that the eigenvalues  for  $T_i(X,\phi)$ are roots of unity for all $i$. Choose a positive integer $N$ such that $\lambda ^N=1$ for any $\lambda \in \bigcup_{i=0}^{2n} \Supp(T_i(X,\phi))$.

Fix an epimorphism $\phi: \pi_1(X) \to \zz $ as before.
Consider the composed surjective map $\pi_1(X) \to \zz \to \zz_N$, where $\zz_N$ is the finite cyclic group with $N$-elements. Denote $X^N$ the corresponding $N$-fold cover of $X$. 

Consider the following diagram, where $\phi_N$ is constructed such that the diagram is commutative.
\begin{equation} \label{diagram}
\xymatrix{
   & & 0 \ar[d] & 0  \ar[d] \\
0 \ar[r] & \pi_1(X^\phi) \ar[d]^\simeq  \ar[r]& \pi_1(X^N) \ar[d] \ar[r]^{\phi_N} &  N\zz \ar[d] \ar[r] & 0\\
0 \ar[r] & \pi_1(X^\phi) \ar[r] &\pi_1(X)  \ar[d]  \ar[r]^{\phi}   & \zz  \ar[d] \ar[r] & 0\\
    &      & \zz_N \ar[r]^\simeq   \ar[d] &     \zz _N  \ar[d] \\
      & & 0  & 0   
}
\end{equation}
In particular, all the horizontal and vertical complexes are short exact sequences. 

Note that $\phi_N: \pi_1(X^N) \to N\zz$ is also an epimorphism. $X^\phi$ can be also viewed as the $N\zz$-infinite cyclic cover of $X^N$ as induced by $\phi_N$, with the deck transformation group $N \zz$. 
Then $H_i(X^\phi)$ becomes a $R_N=\cc[t^N, t^{-N}]$-module, induced by the natural sub-algebra structure $R_N \subset R$. It is indeed the $R_N$-module $H_i(X^\phi) \otimes_{R} R_N$, where $H_i(X^\phi)$ is taken as a $R$-module induced by $\phi$.

   $R_N$ is also a principal ideal domain. The rank of the free part of $H_i(X^\phi)\otimes_{R} R_N$ does not change, and so is the dimension of the torsion part (as a $\cc$-vector space). But now the torsion part has only eigenvalue $1$, if the torsion part is non-zero. Moreover, 
\begin{equation} \label{finte}
\sS_1 \left(\mathrm{Tor}(H_i(X^\phi)\otimes_{R} R_N)\right) = \max_{\lambda \in \Supp(T_i(X,\phi))} \sS_\lambda (T_i(X,\phi)).
\end{equation}

Note that the finite cover $X_N$ of a compact K\"ahler manifold (resp. a smooth complex quasi-projective complex  variety) is still compact K\"ahler (resp. a smooth quasi-projective variety). So  (\ref{finte}) allows us to reduce the study of the maximal size of the Jordan blocks to the case when $\lambda=1$.

\section{Alexander complex and real homotopy type}
In this section, we assume that $X$ is a real manifold. Let $\phi: \pi_1(X)\to \zz$ be an epimorphism. Then the completion of $H^i(X, \sL^\phi)$ as an $R$-module at the maximal ideal $(t-1)$ is determined by the real homotopy type (same as the complex homotopy type) of $X$. In particular, the eigenvalue $1$ part of $T_i(X, \phi)$ is determined by the real homotopy type of $X$ and the $\cc$-linear map $\phi^*: \cc \to H^1(X, \cc) $. In this section, we will make the above two sentences precise and explicit. 

All the results in this section are either folklore or implicit in \cite{dp} and \cite{bw2}. 

The epimorphism $\phi: \pi_1(X)\to \zz$ induces a monomorphism in cohomology $$\phi^{\ast}: \cc \hookrightarrow H^{1}(X,\cc),$$ Let $\eta^\phi$ be a complex valued closed $1$-form, which represents the cohomology class $\phi^*(1_\cc)$.  $\phi$ is omitted for $\eta^\phi$, if the context is clear. We can construct the following thickened complex valued de Rham complex of $X$ along the direction $\eta$,
$$\Omega^\ubul_{DR}(X, \eta, m):=\cdots\to\Omega_{DR}^i(X)\otimes_\cc R/(s^m)\xrightarrow{d\otimes \id+\wedge \eta\cdot s}\Omega_{DR}^{i+1}(X)\otimes_\cc R/(s^m)\to\cdots$$
where $d$ is the differential of the de Rham complex $\Omega_{DR}^\ubul(X)$, and $s=t-1$ is a change of variable. To be more precise, the differential in $\Omega^\ubul_{DR}(X, \eta, m)$ is $\cc$-linear and is defined by
$$\omega\otimes 1\mapsto  d \omega \otimes 1 + (\omega\wedge\eta)\otimes s.$$

First of all, the cohomology of $\Omega^\ubul_{DR}(X, \eta, m)$ does not depend on the closed 1-form $\eta$ we chose to represent $\phi$. 
\begin{prop}\label{independence}
Suppose two closed 1-forms $\eta_1, \eta_2$ represent the same class in $H^1(X, \cc)$. Then $$H^i\left(\Omega_{DR}^\ubul(X, \eta_1, \infty)\right)\cong H^i\left(\Omega_{DR}^\ubul(X, \eta_2, \infty)\right).$$
\end{prop}
\begin{proof}
According to the proof of \cite[Corollary 3.13]{bw2}, one can construct an isomorphism between the two complexes $\Omega_{DR}^\ubul(X, \eta_1, m)$ and $\Omega_{DR}^\ubul(X, \eta_2, m)$, and hence their cohomology groups. Since taking inverse limit is exact in this case, the statement follows. 
\end{proof}

\begin{prop}
Under the above notations, one has the following isomorphism 
\begin{equation}\label{eq1}
H^i\left(X, \sL^\phi\otimes_R R/(s^m)\right)\cong H^i\left(\Omega^\ubul_{DR}(X, \eta, m)\right)
\end{equation}
as $R/(s^m)$-modules. 
\end{prop}
\begin{proof}
In fact, $\Omega^\ubul_{DR}(X, \eta, m)$ is the de Rham resolution (as $R/(s^m)$-modules) of the rank one $R/(s^m)$-local system $\sL^\phi\otimes_R R/(s^m)$ (see \cite[Section 6]{gm} and \cite{bw2}). Hence the isomorphism follows. 
\end{proof}

Denote the completion of $R$ at the maximal ideal $(s)$ by $\widehat{R}$, that is, $\widehat{R}=\displaystyle\varprojlim_m R/(s^m)$. Note that $\widehat{R}= \cc[[s]]$ is a principal ideal domain and a regular local ring. There is a natural $\widehat{R}$-isomorphism for any finitely generated $R$-module $A$,
$$ A \otimes_R \widehat{R} \cong \widehat{A},$$ 
where $\widehat{A}$ is the $(s)$-adic completion of $A$.

Let 
$$\Omega^\ubul_{DR}(X, \eta, \infty):=\cdots\to\Omega_{DR}^i(X)\otimes_\cc \widehat{R}\xrightarrow{d\otimes \id+\wedge \eta\cdot s}\Omega_{DR}^{i+1}(X)\otimes_\cc \widehat{R}\to\cdots$$
be the inverse limit of $\Omega^\ubul_{DR}(X, \eta, m)$ as $m\to +\infty$. Since the left side of the isomorphism (\ref{eq1}) is also functorial on $m$, we can take the inverse limit of the cohomology groups. Since the cohomology groups are finite dimensional $\cc$-vector spaces, the Mittag-Leffler condition is automatic. Thus we have the following. 
\begin{cor} \label{hat} As $\widehat{R}$-modules,
$$H^i(X,\sL^\phi\otimes_R \widehat{R})\cong H^i\left(\Omega^\ubul_{DR}(X, \eta, \infty)\right).$$
\end{cor}
Notice that on the left side, we do not need to take the derived tensor product, since $\widehat{R}$ is a flat $R$-module.

Since $\widehat{R}$ is faithfully flat over the local ring $R_{(s)}$, taking tensor product $\otimes_R \widehat{R}$ preserves the torsion submodule supported at the maximal ideal $(s)$. 
\begin{cor}\label{deRham}  As $\widehat{R}$-modules,
$$   T_i^1(X, \phi)\otimes_R \widehat{R} \cong  \mathrm{Tor} \left( H^{i+1}(X, \sL^\phi\otimes_R \widehat{R}) \right)\cong \mathrm{Tor}\left(H^{i+1}\left(\Omega^\ubul_{DR}(X, \eta, \infty)\right)\right)$$
where ${\rm Tor}$ means taking the torsion part as $\widehat{R}$-modules. 
\end{cor}
\begin{proof}
The first isomorphism comes from  (\ref{uct}) and the fact that the involution does not change the eigenvalue $1$ part. The second isomorphism follows from Corollary \ref{hat}.
\end{proof}

So far, we have given an algorithm to compute $T^1_i(X, \phi)$ only using the de Rham differential algebra of $X$. In fact, this construction applies to any differential algebra and it only depends on the homotopy type of the differential algebra. 

Let $(\sA^\ubul, d)$ be a differential algebra, and let $\eta\in \sA^1$ be a 1-cycle. We define complexes of $R/(s^m)$-modules and $\widehat{R}$-modules respectively,
$$\sA^\ubul(\eta, m)\stackrel{\textrm{def}}{:=}\cdots\to \sA^i\otimes_\cc R/(s^m)\xrightarrow{d\otimes\id+\wedge\eta\cdot s}\sA^{i+1}\otimes_\cc R/(s^m)\to \cdots$$
and
$$\sA^\ubul(\eta, \infty)\stackrel{\textrm{def}}{:=}\cdots\to \sA^i\otimes_\cc \widehat{R}\xrightarrow{d\otimes\id+\wedge\eta\cdot s}\sA^{i+1}\otimes_\cc \widehat{R}\to \cdots$$

\begin{prop}\label{homotopyequiv}
Let $F: \sA^\ubul\to \sB^\ubul$ be a homomorphism of differential algebras, which is a quasi-isomorphism. Let $\eta\in\sA^1$ be a 1-cycle. Then the natural homomorphism $F_*: \sA^\ubul(\eta, \infty)\to \sB^\ubul(F(\eta), \infty)$ induced by $F$ is a quasi-isomorphism of complexes of $\widehat{R}$-modules. 
\end{prop}
\begin{proof}
This follows from a standard spectral sequence argument. See the proof of \cite[Theorem 3.16]{bw2}. 
\end{proof}

\section{Compact K\"ahler manifolds and formality}
A celebrated theorem of Deligne, Griffiths, Morgan and Sullivan claims that the real homotopy type of a compact K\"ahler manifold is formal. Since we only work with complex coefficients, it suffices to know only the complex homotopy type. Let us first recall the definition of a topological space, or a differential algebra, being formal. 

\begin{defn}\label{defnformal}
A differential algebra $(\sA^\ubul, d)$ is called formal, if it can be connected to its cohomology differential algebra $(H^\ubul(\sA^\ubul), 0)$ by a zigzag of quasi-isomorphisms of differential algebras. A smooth manifold is called formal (more precisely, its real homotopy type is formal), if its de Rham complex is formal as a differential algebra. 
\end{defn}
\begin{theorem} \label{dgms}\cite{dgms}
Any connected compact K\"ahler manifold is formal. 
\end{theorem}

\begin{prop} \label{algebraformal}
Let $(\sA^\ubul, 0)$ be a differential algebra with zero differentials. Then for any $\eta\in \sA^1$ and any $i\in\zz_{\geq 0}$, $H^i(\sA^\ubul(\eta, \infty))$ is semi-simple as $\widehat{R}$-module. 
\end{prop}
\begin{proof}
Since the differentials of $\sA^\ubul$ are zero, by definition, 
$$\sA^\ubul(\eta, \infty)=(\sA^\ubul\otimes_\cc \widehat{R}, \wedge\eta\cdot s). $$
The algebra $\widehat{R}=\cc[[s]]$ is a graded algebra, where we let $\deg(s)=1$. Then the complex $\sA^\ubul(\eta, \infty)$ becomes a complex of graded $\widehat{R}$-modules, if we let all elements in $\sA^i$ have degree $0$. Thus, as a complex of free graded $\widehat{R}$-modules, $\sA^\ubul(\eta, \infty)$ is of the following form 
\begin{equation}\label{complex}
\cdots\longrightarrow\widehat{R}^{\oplus n_{i-1}}\overset{\wedge \eta\cdot s}{\longrightarrow} \widehat{R}^{\oplus n_{i}}\overset{\wedge \eta\cdot s}{\longrightarrow} \widehat{R}^{\oplus n_{i+1}}\longrightarrow \cdots
\end{equation} Note that the map $\wedge \eta \cdot s$ is graded of degree $1$. 
It is totally straightforward to check that given any complex of free graded $\widehat{R}$-modules as in (\ref{complex}), its cohomology modules are semi-simple $\widehat{R}$-modules. 
\end{proof}

\begin{proof}[Proof of Proposition \ref{formal}]
The proposition follows immediately from Corollary \ref{deRham}, Proposition \ref{homotopyequiv}, Proposition \ref{algebraformal} and Definition \ref{defnformal}. 
\end{proof}

\begin{proof}[Proof of Theorem \ref{main1}]
Note that the finite (unramified) cover of a compact K\"ahler manifold is still compact K\"ahler. So Proposition \ref{torsion} and (\ref{finte}) allow us to reduce the proof to the case when $\lambda=1$.

When $\lambda=1$, the claim follows from Proposition \ref{formal} and Theorem \ref{dgms}.
\end{proof}

\section{Quasi-projective varieties and Gysin model}

Let $X$ be a smooth complex quasi-projective variety. Fix a good compactification $\overline{X}$ of $X$, i.e.,  we require that the boundary divisor is a simple normal crossing divisor. Let $D=\bigcup_j D_j$ be the boundary divisor $\overline{X}\setminus X$. Following the standard conventions, we let $D_J=\bigcap_{j\in J}D_j$. A result of Morgan \cite{m} says that the rational (and hence the complex) homotopy type of $X$ is determined by the cohomology rings of $D_J$ and the Gysin maps between them. 

To make the above statement more precise, we introduce the Gysin model $(\sA^\ubul, d)=(\sA^\ubul(\overline{X},D), d)$ of $X$ with respect to the compactification $\overline{X}$ (see \cite{m} and \cite{dp}). As graded vector spaces, $\sA^k=\bigoplus_{p+l=k}\sA^{p,l}$ and $\sA^{p, l}=\bigoplus_{|J|=l}H^p(D_J, \cc)$. The product is given by the cup-product
$$H^p(D_J, \cc)\otimes H^{p'}(D_{J'}, \cc)\to H^{p+p'}(D_{J\cup J'}, \cc)$$
which is nonzero only when $J\cap J'=\emptyset$. The differential, $d: \sA^{p,l}\to \sA^{p+2, l-1}$, is a linear map,  which component-wise are Gysin maps of the form 
$$H^p(D_J, \cc)\to H^{p+2}(D_{J'}, \cc)$$
where $J'\subset J$ and $|J'|=|J|-1$. 
\begin{theorem}[\cite{m}]\label{morgan}
Let the Gysin model $(\sA^\ubul, d)$ be defined as above. Then there is a zigzag of quasi-isomorphisms of differential algebras connecting the de Rham complex $\Omega^\ubul_{DR}(X)$ and the Gysin model $(\sA^\ubul, d)$. 
\end{theorem}

Let $\eta=\eta^{1,0}+\eta^{0,1}$ be the decomposition according to $\sA^1=\sA^{0,1}\oplus\sA^{1,0}$. Warning: this is not the Hodge decomposition! In fact, $\eta^{1,0}$ is pure of weight $1$ and $\eta^{0,1}$ is of Hodge type $(1,1)$.  Since the differential vanishes on $\sA^{1,0}$, $\eta^{1,0}$ is a 1-cycle, and hence $\eta^{0,1}$ is also a 1-cycle. Then $\sA^{\ubul,\ubul}$ with $\wedge\eta=\wedge\eta^{0,1}+\wedge\eta^{1,0}$ becomes a double complex. Moreover, the Gysin map $d: \sA^{p,l}\to \sA^{p+2, l-1}$ commutes with both $\wedge\eta^{0,1}$ and $\wedge\eta^{1,0}$ up to a sign. 

By definition, the differentials in the complex $\sA^\ubul(\eta, \infty)$ are given by $$d\otimes \id+\wedge\eta\cdot s=d\otimes \id+\wedge\eta^{0,1}\cdot s+\wedge\eta^{1,0}\cdot s.$$ Thus, we can think of $\sA^\ubul(\eta, \infty)$ as a double complex with three arrows, $\wedge\eta^{0,1}\cdot s$, $\wedge\eta^{1,0}\cdot s$ and $d\otimes \id$. 
Each of them can be considered as a homomorphism of graded free $\widehat{R}$-modules of degree one and zero respectively. Here $\deg (s)=1$ and any element in $\sA^{p,l}$ has degree zero.

Consider the dual double complex $$\sA_\ubul(\eta, \infty):=\homo_{\widehat{R}} \left(\sA^\ubul(\eta, \infty), \widehat{R}\right)$$ with inverted arrows.  Denote $\widehat{\sA}_{p,l} :=\sA_\ubul(\eta, \infty)_{p,l}= \homo_{\widehat{R}} (\sA^{p, l}\otimes_\cc  \widehat{R} , \widehat{R}). $  The corresponding three maps in $\sA_\ubul(\eta, \infty)$, still denoted by $d\otimes \id$, $\wedge\eta^{0,1}\cdot s$ and $\wedge\eta^{1,0}\cdot s$, send the elements in $\widehat{\sA}_{p,l}$ to $\widehat{\sA}_{p-2,l+1}$, $\widehat{\sA}_{p,l-1}$ and $\widehat{\sA}_{p-1,l}$, respectively.    

\begin{prop} \label{homology}  With the above assumptions and notations, 
one has the following $\widehat{R}$-isomorphism:  
$$ H_i\left(\sA_\ubul(\eta, \infty)\right)\cong H_i(X, \sL^\phi \otimes_R \widehat{R}).$$ 
\end{prop}
 
\begin{proof}

$\widehat{R}$ is a PID (principal ideal domain). 
 Using the Universal Coefficient Theorem for PID  (see (\ref{uct})), one can read cohomology from homology, and vice versa. 
 
 Combing Corollary \ref{hat} and Theorem \ref{morgan}, we get that  \begin{center}
$ H^i\left(\sA^\ubul(\eta,\infty)\right) \cong H^i(X, \sL^\phi \otimes_R \widehat{R})$ for all $i$.
 \end{center} Note that $\sA^\ubul(\eta,\infty)$ is a complex of the finitely generated free $\widehat{R}$-modules.  Applying the functor $\homo_{\widehat{R}}(-, \widehat{R})$ to $\sA^\ubul(\eta,\infty)$,  the claim follows from the Universal Coefficient Theorem (see \ref{uct}) and the fact that the involution does not change the eigenvalue $1$ part.
 \end{proof}

\begin{prop}   With the above assumptions and notations, the eigenvalue $1$ torsion part $T^1_i(X,\phi)\otimes_R \widehat{R}$  is annihilated by $s^{\min\{i+1, 2n-i\}}$.
\end{prop}
\begin{proof}
 Now consider the filtrations by columns for the total complex $\sA_\ubul(\eta, \infty)$.  Then the zeroth page of the associated spectral sequence is 
$$E^0_{p, q} \cong \widehat{\sA}_{p, q}$$
with $d^0=\wedge\eta^{0,1}\cdot s$. There are two hidden arrows $\wedge\eta^{1,0}\cdot s$ and $d\otimes \id$, which will be reflected in later pages and which are both homomorphisms of graded $\widehat{R}$-modules.

On the first page of the spectral sequence, 
$$E_{p, q}^1=H_q(\widehat{\sA}_{p, \ubul}  ,  \wedge\eta^{0,1}\cdot s)$$
with $d^1=\wedge\eta^{1, 0}\cdot s$. It is straightforward to check that $E_{p,q}^1$ are graded $\widehat{R}$-modules and are generated by elements of degree zero. In fact, since $d^0$ is graded of degree one, $E_{p,q}^1$ are all semi-simple as $\widehat{R}$-modules. Since $d\otimes \id$ commutes with $\wedge\eta^{0,1}$ up to a sign, the hidden arrows $d\otimes \id$ induces arrows on the $E^1$ page, which are homomorphisms of graded $\widehat{R}$-modules of degree zero. 

Consider the $E^1$ page differential $$d^1=\wedge\eta^{1, 0}\cdot s: E_{p, q}^1 \to E_{p-1, q}^1.$$ 
Notice that for all $q$, $E_{p, q}^1=H_q(\widehat{\sA}_{p, \ubul}  ,  \wedge\eta^{0,1}\cdot s)$ are semi-simple $\widehat{R}$-modules generated by elements in degree zero. Since $d^1$ is graded of degree one, the torsion part of $E_{p, q}^1$ is contained in the kernel of $d^1$. Moreover, the induced map from the free part of $E_{p, q}^1$ to the torsion part of $E_{p-1, q}^1$ is also trivial. Therefore, the kernel of $d^1$ is generated by elements of degree zero. Thus $E_{p,q}^2$ are generated by elements of degree zero. Similarly, we can easily see that the cokernels of $d^1$ are also semi-simple. Since submodules of semi-simple $\widehat{R}$-modules are semi-simple, $E_{p,q}^2$ are semi-simple $\widehat{R}$-modules. 

Now we have proved that $E_{p,q}^2$ are all semi-simple $\widehat{R}$-modules generated by elements in degree zero. The differential $d^2$ is induced by $d\otimes \id$. Hence $d^2$ is  graded  of degree zero. It is easy to check that given any complex of semi-simple graded $\widehat{R}$-modules generated in degree zero, where the differentials are graded of degree zero, the cohomology modules are also semi-simple. Therefore, $E_{p,q}^3$ are semi-simple $\widehat{R}$-modules. 

Clearly the spectral sequence degenerates at $E^3$. The spectral sequence induces a filtration on $H_{i}\left(\sA_\ubul(\eta, \infty)\right)$, such that the graded quotients are semi-simple $\widehat{R}$-modules. Notice that $E_{p,q}^0=0$ if $p<0$ or $q<0$ or $p+2q>2n$. Therefore, the number of steps of the filtration on $H_{i}\left(\sA_\ubul(\eta, \infty)\right)$ is at most $\min\{i+1, 2n-i+1\}$. Note that, when $p+2q=2n$, $E^3_{p, q}$ is an $\widehat{R}$-submodule of $E^0_{p, q}$, and hence free.  So the bound can be improved from $\min\{i+1, 2n-i+1\}$ to $\min\{i+1, 2n-i\}$.

Now, the claim follows from Proposition \ref{homology}.
\end{proof}

\begin{proof}[Proof of Theorem \ref{main2}]
Note that the finite (unramified) cover of a smooth complex quasi-projective  variety is still a smooth variety. So Proposition \ref{torsion} and (\ref{finte}) allow us to reduce the proof to the case when $\lambda=1$, which is confirmed by the preceding proposition. 
\end{proof}

\begin{proof}[Proof of Proposition \ref{pure}]
When $\lambda=1$, the bound $\min\{i+1, 2n-i\}$ can be slightly improved if $\phi$  is pure with weight 2. In this case,  $\eta^{1,0}=0$. Hence, when $p+2q=2n-1$, $E^3_{p, q}$ is an $\widehat{R}$-submodule of $E^0_{p, q}$ and hence free.  Thus the bound $\min\{i+1, 2n-i\}$ can be improved into $\min\{i+1, 2n-i-1\}$.

Note that we can assume that the $N$-fold cover $X^N$ is a smooth quasi-projective variety and the covering map $c: X^N \to X$ is algebraic.  Since $X$ is smooth, $Rc_* \cc_{X^N}$ is a rank $N$ $\cc$-local system on $X$. The corresponding representation is as follows: 
$$ \pi_1(X) \overset{\phi}{\to} \zz \to GL_N(\cc),$$
where the last map sends $1_\zz$ to the diagonal matrix with the main diagonal elements $\{1, e^{2 \pi i/N}, \cdots, e^{2 \pi i (N-1)/N}  \}.$ 
In particular, this local system splits with $N$-direct summands, and one of them is the constant sheaf $\cc_X$. This shows that the covering map induces an injective mixed Hodge structure map $$H^1(X) \hookrightarrow H^1(X^N).$$ Moreover, one has the following commutative maps associated to the diagram (\ref{diagram}): $$
\xymatrix{
H_1(X^N)   & \homo (\zz, \cc)= \cc\ar[l]_{\phi_N^*}  \\
H_1(X)    \ar[u]    & \homo (\zz, \cc)=\cc  \ar[l]_{\phi^*}  \ar[u]_{N^*} 
} 
$$ where $N^*$ is the $N$-scaling linear isomorphism for $\cc$.   So, if $\phi$ is pure type with weight two, then so is  $\phi_N$ for $X^N$. Therefore, the claim for the case $\lambda\neq 1$ can be reduced to the case $\lambda=1$ by Proposition \ref{torsion} and (\ref{finte}).
\end{proof}

\end{document}